\documentclass[11pt]{amsart}
\usepackage{amsmath,amssymb,amsthm,times}
\usepackage{mathtools,mathrsfs}
\usepackage{microtype}
\usepackage[colorlinks=true]{hyperref}

\numberwithin{equation}{section}
\theoremstyle{plain}
\newtheorem{theorem}{Theorem}[section]
\newtheorem{lemma}[theorem]{Lemma}
\newtheorem{proposition}[theorem]{Proposition}
\newtheorem{corollary}[theorem]{Corollary}
\theoremstyle{definition}

\theoremstyle{remark}
\newtheorem{remark}[theorem]{Remark}

\allowdisplaybreaks[4]
\newcommand{\dd}{\,\mathrm{d}}

\newcommand{\R}{\mathbb R}
\newcommand{\Hrad}{H^1_{\mathrm{rad}}}
\newcommand{\norm}[2]{\lVert #1\rVert_{#2}}
\begin{document}
\title[On Sirakov's equal-frequency uniqueness conjecture]{On Sirakov's equal-frequency uniqueness conjecture}


\author[H.G. Chen]{Hong-Ge Chen}
\address{Hong-Ge Chen, School of Mathematics and Statistics, Key Laboratory of Nonlinear Analysis \& Applications (Ministry of Education),
Central China Normal University, Wuhan, China}
\email{hongge\_chen@whu.edu.cn}

\author[Y. Liu]{Yong Liu}
\address{Yong Liu, School of Mathematics and Statistics, Beijing Technology and Business University, Beijing, China}
\email{yliumath@btbu.edu.cn}

\author[J.C. Wei]{Juncheng Wei}
\address{Juncheng Wei, Department of Mathematics, Chinese University of Hong Kong, Shatin, New Territories, Hong Kong}
\email{wei@math.cuhk.edu.hk}


\author[W. Yang]{Wen Yang}
\address{Wen Yang, Department of Mathematics, Faculty of Science, University of Macau, Taipa, Macau, China}
\email{wenyang@um.edu.mo}

\begin{abstract}
Let $N\in\{2,3\}$, $0<\mu _1\leq\mu _2$, and $0<\beta<\mu _1$.  We prove
that the equal-frequency two-component cubic Schr\"odinger system
\[
 -\Delta u+u=\mu _1u^3+\beta uv^2,
 \qquad
 -\Delta v+v=\mu _2v^3+\beta u^2v
 \quad\text{in }\R^N
\]
has exactly one positive solution in $H^1(\R^N)\times H^1(\R^N)$ modulo simultaneous translations.  More precisely, every positive solution is a simultaneous translate of the synchronized state constructed from the unique positive radial solution of $-\Delta w+w=w^3$ in $\R^N$.  This settles Sirakov's equal-frequency uniqueness conjecture throughout the weak-coupling range. 

The main difficulty in the proof is to exclude radial solutions for which the ratio of the normalized components is nonconstant.  After normalization, the two components satisfy scalar equations with a common potential. We construct a weighted Pohozaev functional for the system together with a correction term and prove that both the corrected functional and the associated weighted functional are strictly positive. Combining these sign properties with a radial flux identity and an auxiliary quotient associated with the component ratio forces synchronization.
\end{abstract}


\maketitle

\section{Introduction and main results}

Coupled nonlinear Schr\"odinger equations arise in various different contexts such as the study of
multicomponent Bose--Einstein condensates and nonlinear optical media, where
the components represent interacting condensates or optical modes and the
coupling coefficient measures the intercomponent interaction.  When the two
linear frequencies agree, a natural rescaling leads to the stationary system
\begin{equation}\label{1-1}
 \begin{cases}
  -\Delta u+u=\mu _1u^3+\beta uv^2,\\
  -\Delta v+v=\mu _2v^3+\beta u^2v,
 \end{cases}
 \qquad x\in\R^N.
\end{equation}
Here $\mu _1,\mu _2>0$ are the self-interaction coefficients and $\beta$ is
the coupling coefficient.  For $N\leq3$, system \eqref{1-1} is the
Euler--Lagrange system of
\[
 \begin{aligned}
 \mathcal E(u,v)
 ={}&\frac12\int_{\R^N}
   \bigl(|\nabla u|^2+u^2+|\nabla v|^2+v^2\bigr)\,\dd x\\
 &-\frac14\int_{\R^N}
   \bigl(\mu _1u^4+2\beta u^2v^2+\mu _2v^4\bigr)\,\dd x .
 \end{aligned}
\]
Thus positive $H^1$ solutions are finite-energy standing-waves. It turns out that
their structure depends sensitively on the sign and size of the coupling.

For equation \eqref{1-1}, one may consider the following three types of questions: (i) existence of positive solutions; (ii) existence and uniqueness of least-energy solutions; and (iii) uniqueness among all positive solutions. The first two questions have been studied extensively by variational methods. For finite cubic systems in dimensions at most three, Lin and Wei proved several sufficient conditions for the existence or nonexistence of ground states in terms of the signs of the couplings and positive-definiteness properties of an associated matrix \cite{LinWeiGroundState,LinWeiGroundStateErratum}. Sirakov identified parameter regimes for the existence or nonattainment of least-energy two-component states and, in specified intermediate regimes, for the nonexistence of nonstandard nonnegative solutions \cite{Sirakov}. For unequal frequencies under specified orderings of the linear and self-interaction coefficients, Chen and Zou identified an optimal coupling threshold for the existence of least-energy solutions \cite[Theorem~1.2 and Remarks~1.1--1.3]{ChenZou}. Related work studies least-energy spike solutions in bounded domains \cite{LinWeiSpikes} and concentration under spatially varying trapping potentials \cite{LinWeiTrapping}. It also includes positive multi-bump bound states for weak repulsive coupling \cite{LinWeiSolitary} and least-energy symbiotic bright solitons in the self-defocusing, strongly cross-attractive regime \cite{LinWeiSymbiotic}. Such existence and ground-state results do not by themselves classify all positive solutions of the autonomous equal-frequency system \eqref{1-1} and a stronger uniqueness question will be the central subject of this paper.

We now return to the autonomous equal-frequency problem and assume, without
loss of generality, that $0<\mu _1\leq\mu _2$.  Let $w$ denote the unique
positive radial solution of
\[
 -\Delta w+w=w^3\qquad\text{in }\R^N.
\]
It is classical that this equation admits a unique positive solution which decays exponentially at infinity; see \cite{BerestyckiLionsPeletier,Kwong}. For $\beta>0$ and $\beta\notin[\mu_1,\mu_2]$, the pair
\begin{equation}
\label{11}
 \left(
  \sqrt{\frac{\mu _2-\beta}{\mu _1\mu _2-\beta^2}}\,w,
  \sqrt{\frac{\mu _1-\beta}{\mu _1\mu _2-\beta^2}}\,w
 \right)
\end{equation}
is a positive synchronised solution of \eqref{1-1}.  Sirakov \cite[Remark~2]{Sirakov} conjectured that, in this equal-frequency setting, this is the unique positive solution.  We refer to this statement as Sirakov's equal-frequency uniqueness conjecture.  It is distinct from the so-called Sirakov open problem concerning the optimal existence range when the two linear frequencies are unequal; see \cite[Section~1.3]{WeiZhongZou}.  The equal-frequency uniqueness conjecture naturally splits into the strong-coupling range $\beta>\mu _2$ and the weak-coupling range $0<\beta<\mu _1$. The conjecture is shown to be true in the strong-coupling range by Wei and Yao \cite{WeiYao}.

For sufficiently small positive coupling, Ikoma established uniqueness modulo translations in a more general constant-coefficient framework encompassing \eqref{1-1} \cite[Theorem~1.1]{Ikoma}. Wei and Yao later provided an alternative perturbative proof for small coupling, established the equal-frequency conjecture in the strong-coupling regime, and completed the positive-coupling classification in dimension one \cite{WeiYao}. Under the additional assumption $\mu _1\ne\mu _2$, Chen and Zou proved uniqueness for weak couplings sufficiently close to $\min\{\mu _1,\mu _2\}$ from below \cite[Theorem~1.1]{ChenZou}. Zhou and Wang subsequently proved the analogous near-endpoint result in the symmetric case $\mu _1=\mu _2=\mu$, $N\in\{2,3\}$, and for the related system, see \cite[Theorems~1.1 and~1.2]{ZhouWang}. Mandel developed uniqueness criteria for a broader class of component-symmetric semilinear elliptic systems, while in the equal-self-interaction cubic specialisation these criteria recover the corresponding Wei--Yao ranges \cite[Corollary~1.4]{Mandel}. Thus, in both the symmetric and nonsymmetric cases, the previously known weak-coupling results covered sufficiently small coupling and a neighbourhood of the upper endpoint, but left an intermediate interval untreated. The present theorem fills this interval and thereby covers the full range $0<\beta<\min\{\mu _1,\mu _2\}$. For further discussion, see Bartsch--Zhong--Zou
\cite[Remark~2.6]{BartschZhongZou} and Wei--Zhong--Zou \cite[Section~5]{WeiZhongZou}.

  Here and below, simultaneous translation by $x_0\in\R^N$ means
$(u,v)\mapsto(u(\,\cdot-x_0),v(\,\cdot-x_0))$. For a radial function about the origin we write $f(x)=f(r)$, where
$r:=|x|$.  We also set
\[ \Hrad(\R^N):=\{f\in H^1(\R^N):~f(x)=f(|x|)\}. \]

The remaining uniqueness question for positive solutions in the
weak-coupling range is: 
\begin{center}
\bfseries
For $N\in\{2,3\}$ and $0<\beta<\min\{\mu _1,\mu _2\}$, does every
positive solution of \eqref{1-1} necessarily take the synchronised
form?
\end{center}
The present paper answers this question affirmatively.  Our main result in this paper is the following classification theorem:

\begin{theorem}
\label{theorem1-1}
Let $N\in\{2,3\}$, $0<\mu _1\leq\mu _2$, 
$0<\beta<\mu _1$.  If $(u,v)\in\Hrad(\R^N)\times\Hrad(\R^N)$ is a positive  solution of \eqref{1-1}, then necessarily
\begin{equation}\label{1-4}
 (u,v)=
 \left(
  \sqrt{\frac{\mu _2-\beta}{\mu _1\mu _2-\beta^2}}\,w,
  \sqrt{\frac{\mu _1-\beta}{\mu _1\mu _2-\beta^2}}\,w
 \right),
\end{equation}
where $w\in\Hrad(\R^N)$ is the unique positive solution of
\begin{equation}\label{1-5}
 -\Delta w+w=w^3\qquad\text{in }\R^N.
\end{equation} 
\end{theorem}

\begin{remark}
Note that the pair in \eqref{1-4} is indeed a positive radial solution of \eqref{1-1}. 
The space $\Hrad(\R^N)$ fixes the symmetry centre at the origin.  Thus
Theorem~\ref{theorem1-1} gives exactly one positive radial solution in the radial class. Moreover, any positive solution $  (u,v)  $ in $  H^1_{\rm rad}(\mathbb{R}^N)\times H^1_{\rm rad}(\mathbb{R}^N)  $ consists of $  C^\infty  $ positive functions. We may therefore restrict attention to classical solutions of \eqref{1-1} throughout the paper.
\end{remark}

The radiality hypothesis can be removed due to the classical result of \cite{BuscaSirakov}. Specifically, we have

\begin{corollary}
\label{cor1-3}
Let $N\in\{2,3\}$, $0<\mu _1\leq\mu _2$, $0<\beta<\mu _1$.  Every
positive solution
$(u,v)\in H^1(\R^N)\times H^1(\R^N)$ of \eqref{1-1} has, for some
$x_0\in\R^N$, the form
\begin{equation}\label{1-6}
 (u(x),v(x))=
 \left(
  \sqrt{\frac{\mu _2-\beta}{\mu _1\mu _2-\beta^2}}\,w(x-x_0),
  \sqrt{\frac{\mu _1-\beta}{\mu _1\mu _2-\beta^2}}\,w(x-x_0)
 \right),
\end{equation}
where $w$ is the same solution as in Theorem \ref{theorem1-1}.
\end{corollary}
The conclusion concerns all positive $H^1$ solutions, not merely ground states or solutions constructed by a variational construction. This distinction is essential, since uniqueness of a ground state does not exclude additional positive solutions at higher energy. Combined with the strong-coupling theorem of Wei and Yao \cite{WeiYao}, Corollary~\ref{cor1-3} yields equal-frequency uniqueness for every positive coupling outside $[\mu _1,\mu _2]$ in dimensions two and three. 
While for $\beta\in\left[\mu_1,\mu_2\right]$ (here we have assumed that $\mu_1<\mu_2$), \cite{bwwei} proved that there is no positive solution; for $N=3$,
compare also \cite[Remark~2.6]{BartschZhongZou}. If $\mu _1=\mu _2=\beta=\mu$, uniqueness fails, since the one-parameter family
\[  \mu^{-\frac12}(\cos\vartheta\,w,\sin\vartheta\,w),
 \qquad 0<\vartheta<\frac{\pi}{2},
\]
consists of positive solutions. Finally, when $\beta=0$, the system decouples and the two scalar components may be translated independently,
so uniqueness up to simultaneous translation also fails. 

We now describe the main ideas of the proof.  The main difficulty is to
establish synchronization in the radial class.  The direct comparison
identity for the two equations has no fixed sign, and an ordinary scalar
Pohozaev identity does not control the relative size of the components.  Our
argument combines a positive system Pohozaev functional with two
differential identities for the component ratio.

The first step is a normalization adapted to the canonical synchronized
state.  With the constants in \eqref{3-1}, write
$u=\lambda _1y_1$ and $v=\lambda _2y_2$.  The two components then satisfy
\[
 -\Delta y_i+\mathcal V y_i=\delta y_i^3,
 \qquad
 \mathcal V=1-\varepsilon P,
 \qquad
 P=\theta _1y_1^2+\theta _2y_2^2,
 \qquad i\in\{1,2\}.
\]
Here $0<\delta<1$, $\varepsilon=1-\delta>0$, and
$\theta _1+\theta _2=1$. Thus, the two components have the same
self-consistent potential and differ only through their cubic terms.

The second step is a weighted Pohozaev identity adapted to this simultaneous potential.  Recall that generalized radial Pohozaev identities for scalar equations were developed by Kawano--Ni--Yotsutani \cite{KawanoNiYotsutani} and Shioji--Watanabe \cite{ShiojiWatanabe}. Here we need to deal with a system that is much more complicated. 

Set $m=N-1$ and choose the weights $a,b,c$ as in \eqref{4-1}.  Our first key observation is that the functions $J_i$ defined in \eqref{4-2} satisfy
\[  J_i'=G y_i^2 \]
with the same function $G$ for both components.  The weighted sum and its correction,
\[  \mathscr J=\theta _1J_1+\theta _2J_2,  \qquad  \mathscr K=\mathscr J-\frac{\varepsilon a}{4}P^2, \]
satisfy the identity
\[  \mathscr K'(r)=\frac{ma(r)P(r)}{3r}\left(\frac{(3-m)(2m-3)}{9r^2}-1\right). \]
The correction cancels the terms containing the derivative of the self-consistent potential.  Combining this identity with the asymptotic
expansions at the origin and the exponential decay at infinity gives 
\[   \mathscr K(r)>0, \qquad \mathscr J(r)>0  \quad\text{for every }r>0. \]
We then use $m\in\{1,2\}$, equivalently
$N\in\{2,3\}$ to get suitable sign for involved functions.

The construction has a feature that may be useful beyond the present system. Its key ingredient is the cancellation principle: once the
equations are written with a common self-consistent potential, the weights and correction are chosen so that derivatives of that potential cancel, leaving a one-dimensional expression whose sign can be tested. This suggests a possible strategy for other cooperative systems admitting a comparable common-potential reduction. The exact cancellation used here depends on the two-component cubic structure and the common linear frequency, while the positivity argument uses $m=N-1\in\{1,2\}$.  Extensions beyond this setting would therefore require a new common-potential reduction and a separate sign analysis. This should be another very interesting problem.

It remains to convert this positivity into synchronisation.  Suppose first that $Y_2:=y_2(0)>Y_1:=y_1(0)$, and set
$\eta=y_2/y_1$.  The normalised equations give 
\[
 \bigl(r^my_1^2\eta'\bigr)'
 =-\delta r^my_1^4\eta(\eta^2-1),
\]
so $\eta$ initially decreases.  On the other hand, the auxiliary quotient
$Z$ introduced in \eqref{5-6} satisfies
\[ Z'=\frac{2\eta\eta'}{(\theta _1+\theta _2\eta^2)^2}\,\mathscr J.
\]
Since $\mathscr J>0$, the signs of $Z'$ and $\eta'$ agree.  If $\eta'$ had a
first zero, the flux identity would force  $\eta$ to be less than one there.  The exact formula for $Z$ would then give $Z>0$, whereas
 $\lim_{r\to 0}Z(r)=0$ and $Z'<0$ before the first zero would give
 $Z<0$.  If no such zero existed, $Z$ would decrease strictly from zero, contradicting the fact that $\lim_{r\to \infty}Z(r)=0$.  Hence $Y_2>Y_1$ is impossible; interchanging the
components also excludes $Y_1>Y_2$.  Equality of the central values then
implies $y_1\equiv y_2$.

After synchronisation, the common profile solves
$-\Delta w+w=w^3$.  The uniqueness of this solution then follows from
Kwong's theorem \cite[Theorem, pp.~265--266]{Kwong}. Returning to the original variables yields \eqref{1-4}. By the classical moving-plane result of Busca and Sirakov \cite[Section~2.1, Theorem~2]{BuscaSirakov}, the conclusion of Theorem~\ref{theorem1-1} extends to any positive  solution of \eqref{1-1}. We may therefore restrict attention to radially symmetric positive solutions throughout the paper.

The paper is organized as follows. Section~\ref{section3} introduces the common-potential formulation and derives the asymptotic estimates needed later. Section~\ref{section4} constructs the weighted system Pohozaev functional and establishes its positivity. Section~\ref{section5} combines this positivity with the component-ratio identities to prove synchronisation. Finally, Section~\ref{section6} contains the proofs of Theorem~\ref{theorem1-1} and Corollary~\ref{cor1-3}.

\section{The common-potential normalisation}\label{section3}

In this section, we transform the original system into a pair of new equations which have the same potentials, which enables us to compare the two solutions using ODE analysis.

To begin with, let us introduce some notation and define
\begin{equation}\label{3-1}
 \begin{gathered}
  D:=\mu _1\mu _2-\beta^2,
  \qquad
  \lambda _1:=\sqrt{\frac{\mu _2-\beta}{D}},
  \qquad
  \lambda _2:=\sqrt{\frac{\mu _1-\beta}{D}},\\
  \delta:=\frac{(\mu _1-\beta)(\mu _2-\beta)}{D},
  \qquad
  \varepsilon:=\frac{\beta(\mu _1+\mu _2-2\beta)}{D},\\
  \theta _1:=\frac{\mu _2-\beta}{\mu _1+\mu _2-2\beta},
  \qquad
  \theta _2:=\frac{\mu _1-\beta}{\mu _1+\mu _2-2\beta}.
 \end{gathered}
\end{equation}

\begin{lemma}\label{lemma3-1}
Let $N\in\{2,3\}$ and $0<\beta<\min\{\mu _1,\mu _2\}$.  Then all quantities in
\eqref{3-1} are well-defined and satisfy
\begin{equation}\label{3-2}
 D>0,
 \quad 0<\delta<1,
 \quad \varepsilon=1-\delta>0,
 \quad \theta _1,\theta _2>0,
 \quad \theta _1+\theta _2=1.
\end{equation}
Moreover,
\begin{equation}\label{3-3}
 \begin{gathered}
  \varepsilon\theta _1=\beta\lambda _1^2,
  \qquad
  \varepsilon\theta _2=\beta\lambda _2^2,\\
  \delta+\varepsilon\theta _1=\mu _1\lambda _1^2,
  \qquad
  \delta+\varepsilon\theta _2=\mu _2\lambda _2^2.
 \end{gathered}
\end{equation}
For any $(u,v)\in H^1(\R^N)\times H^1(\R^N)$, define
\begin{equation}\label{3-4}
 \begin{gathered}
  y_1(x):=\frac{u(x)}{\lambda _1},
  \qquad
  y_2(x):=\frac{v(x)}{\lambda _2},\\
  P(x):=\theta _1y_1(x)^2+\theta _2y_2(x)^2,
  \qquad
  \mathcal V(x):=1-\varepsilon P(x).
 \end{gathered}
\end{equation}
Then $(u,v)$ is a solution of \eqref{1-1} if and only if
$(y_1,y_2)\in H^1(\R^N)\times H^1(\R^N)$ is a solution of
\begin{equation}\label{3-5}
 -\Delta y_i+\mathcal V y_i=\delta y_i^3,
 \qquad x\in\R^N,\quad i\in\{1,2\}.
\end{equation}
\end{lemma}

\begin{proof}
Since $\mu _1>\beta$ and $\mu _2>\beta$, one has
$D=\mu _1\mu _2-\beta^2>0$. It follows that $\delta>0$, and
\[
 D-(\mu _1-\beta)(\mu _2-\beta)
 =\beta(\mu _1+\mu _2-2\beta)>0.
\]
This proves $0<\delta<1$ and the asserted formula for $\varepsilon$.
The conclusions about the $\theta_i$ follow from their definitions.
For example,
\[
 \varepsilon\theta _1
 =\frac{\beta(\mu _1+\mu _2-2\beta)}D
   \frac{\mu _2-\beta}{\mu _1+\mu _2-2\beta}
 =\beta\lambda _1^2,
\]
and
\[
 \delta+\varepsilon\theta _1
 =\frac{(\mu _1-\beta)(\mu _2-\beta)+\beta(\mu _2-\beta)}D
 =\frac{\mu _1(\mu _2-\beta)}D
 =\mu _1\lambda _1^2.
\]
The two identities with index $2$ follow by the same calculation.

Substitution of $u=\lambda _1y_1$ and $v=\lambda _2y_2$ into the first
equation of \eqref{1-1}, followed by division by $\lambda _1$, gives
\[
 -\Delta y_1+y_1
 =\mu _1\lambda _1^2y_1^3+\beta\lambda _2^2y_1y_2^2.
\]
By \eqref{3-3}, the right-hand side equals
\[
 (\delta+\varepsilon\theta _1)y_1^3
 +\varepsilon\theta _2y_1y_2^2
 =\delta y_1^3+\varepsilon P y_1.
\]
Moving the last term to the left gives \eqref{3-5} for
$i=1$.  The second original equation similarly becomes
\[
 -\Delta y_2+y_2
 =\varepsilon\theta _1y_1^2y_2
  +(\delta+\varepsilon\theta _2)y_2^3
 =\varepsilon P y_2+\delta y_2^3.
\]
Since $\lambda _1,\lambda _2>0$, the converse follows by reversing the
calculation.
\end{proof}

Let $m:=N-1\in\{1,2\}$.  The radial form of
\eqref{3-5} is
\begin{equation}\label{3-6}
 y_i''+\frac mr y_i'-\mathcal V y_i+\delta y_i^3=0,
 \qquad r>0,\quad i\in\{1,2\}.
\end{equation}
In what follows, dependence on $r$ is omitted when no confusion can arise.
Unless stated otherwise, all $O(\cdot)$ estimates are taken as $r\to 0$
and are uniform for $i\in\{1,2\}$.

\begin{lemma}
\label{lemma3-2}
Let $N\in\{2,3\}$ and $0<\beta<\min\{\mu _1,\mu _2\}$, and define
$\delta,\varepsilon,\theta _1,\theta _2$ by
\eqref{3-1}.  Suppose that
$(y_1,y_2)\in\Hrad(\R^N)\times\Hrad(\R^N)$ is a positive solution of
\eqref{3-5}, where
$P=\theta _1y_1^2+\theta _2y_2^2$ and $\mathcal V=1-\varepsilon P$.
Then $y_i\in C^\infty(\R^N)$.  Let $Y_i:=y_i(0)>0$.  The following expansions hold as $r\to 0$:
\begin{equation}\label{3-7}
\begin{aligned}
& y_i(r)=Y_i+\frac{\mathcal V(0)Y_i-\delta Y_i^3}{2N}r^2+O(r^4),\\
& y_i'(r)=\frac{\mathcal V(0)Y_i-\delta Y_i^3}{N}r+O(r^3).
 \end{aligned}
\end{equation}
 Moreover, there exist constants $C,\kappa>0$, depending on
the parameters and the solution, such that
\begin{equation}\label{3-8}
 |y_i(r)|+|y_i'(r)|\leq Ce^{-\kappa r}\qquad \mbox{for}~
r\geq1,\quad i\in\{1,2\}.
\end{equation}
\end{lemma}

\begin{proof}
By Lemma~\ref{lemma3-1}, the functions
$u:=\lambda _1y_1$ and $v:=\lambda _2y_2$ form a positive solution of
\eqref{1-1}. Classical elliptic regularity theory then yields that each
$y_i$ is smooth and strictly positive in $\R^N$. With $e_1:=(1,0,\ldots,0)\in\R^N$, the
one-variable function $t\mapsto y_i(te_1)$ is smooth and even, so its first 
and third derivatives vanish at zero.  Furthermore,
$\Delta y_i(0)=Ny_i''(0)$, where the double prime denotes the derivative of
the radial profile.  Evaluating \eqref{3-5} at the origin gives
\[
 y_i''(0)=\frac{\mathcal V(0)Y_i-\delta Y_i^3}{N}.
\]
Consequently, we get \eqref{3-7}.

By Strauss's radial lemma~\cite[Radial Lemma~1]{Strauss}, we have $y_i(r)\to0$ as $r\to\infty$.
Hence
\[
 q_i(r):=\mathcal V(r)-\delta y_i(r)^2
 =1-\varepsilon P(r)-\delta y_i(r)^2\to1\qquad \mbox{as}~~r\to\infty.
\]
Define $z_i(r):=r^{\frac{m}{2}}y_i(r)$.  Differentiating
$y_i(r)=r^{-\frac{m}{2}}z_i(r)$ and using \eqref{3-6} yields
\begin{equation}\label{3-9}
 z_i''(r)=Q_i(r)z_i(r),
 \qquad
 Q_i(r):=q_i(r)-\frac{m(2-m)}{4r^2}.
\end{equation}
Since $Q_i(r)\to1$ as $r\to\infty$, there are $R\geq1$ and constants
$0<q_0\leq q_1<\infty$ such that
\begin{equation}\label{3-10}
 0<q_0\leq Q_i(r)\leq q_1\quad\mbox{for}~~
r\geq R,\ i\in\{1,2\}.
\end{equation}
Moreover, $z_i(r)>0$ for $r>0$.  Since $y_i\in L^2(\R^N)$ and it is radial, we have
\begin{equation}\label{3-11}
 \begin{aligned}
  \int_R^\infty |z_i(r)|^2\dd r
  &=\int_R^\infty |y_i(r)|^2r^{N-1}\dd r\leq\frac{1}{|\mathbb S^{N-1}|}
    \norm{y_i}{L^2(\R^N)}^2<\infty.
 \end{aligned}
\end{equation}
Thus
$y_i\in L^2((0,\infty),r^{N-1}\,dr)$ and
$z_i\in L^2((R,\infty),\,dr)$.

In $[R,\infty)$, $z_i''=Q_i z_i>0$, so $z_i'$ is strictly increasing. If $z_i'(r_0)\geq0$ at some point, then $z_i$ is nondecreasing on
$[r_0,\infty)$ and cannot belong to $L^2(R,\infty)$ since $z_i(r_0)>0$.  Hence $z_i'<0$. Thus $z_i(r)\to L_i$ for some
$L_i\geq0$; if $L_i>0$, then $z_i\notin L^2(R,\infty)$, so $L_i=0$. Also $\ell_i:=\lim_{r\to\infty}z_i'(r)\leq0$ exists; if $\ell_i<0$,
then $z_i$ eventually becomes negative. As a result, $\ell_i=0$, and 
\begin{equation}\label{3-12}
 z_i(r)\to0,
 \qquad z_i'(r)\to0
 \quad\text{as }r\to\infty.
\end{equation}

Set $E_i(r):=z_i'(r)^2-q_0z_i(r)^2$.  From \eqref{3-9} and
\eqref{3-10}, we have
\[
 E_i'(r)=2z_i'(r)\bigl(Q_i(r)-q_0\bigr)z_i(r)\leq0.
\]
Since $E_i'(r)\leq0$ and $E_i(r)\to0$ as $r\to\infty$, we deduce that
$E_i(r)\geq0$ for $r\geq R$. It follows from  $z_i'(r)<0$ that
$ -z_i'(r)\geq\sqrt{q_0}\,z_i(r)$. Integration yields
\[
 z_i(r)\leq z_i(R)e^{-\sqrt{q_0}(r-R)}.
\]
Integrating $z_i''(s)=Q_i(s)z_i(s)$ over $(r,T)$ and letting $T\to\infty$
gives
\[
 -z_i'(r)=\int_r^\infty Q_i(s)z_i(s)\dd s
 \leq q_1\int_r^\infty z_i(s)\dd s
 \leq C e^{-\sqrt{q_0}r}.
\]
Since $y_i(r)=r^{-\frac{m}{2}}z_i(r)$ and
$y_i'(r)=r^{-\frac{m}{2}}z_i'(r)-\frac{m}{2}r^{-\frac{m}{2}-1}z_i(r)$, choosing
$0<\kappa<\sqrt{q_0}$ and absorbing the polynomial factors into the constant
gives \eqref{3-8} on $[R,\infty)$.  Increasing $C$ if necessary, we are able to extend the estimate
to $[1,\infty)$.
\end{proof}

\section{New Pohozaev type functions and their properties}\label{section4}

For $m\in\{1,2\}$ and $r>0$, define
\begin{equation}\label{4-1}
 a(r):=r^{\frac{4m}{3}},
 \qquad
 b(r):=\frac{m}{3r}a(r),
 \qquad
 c(r):=\frac{m(3-m)}{9r^2}a(r).
\end{equation}
For $i\in\{1,2\}$, we set
\begin{equation}\label{4-2}
 \begin{split}
 J_i(r):={}&\frac{a(r)}2y_i'(r)^2+b(r)y_i(r)y_i'(r)
      +\frac{c(r)}2y_i(r)^2\\
 &-\frac{a(r)}2\mathcal V(r)y_i(r)^2
      +\frac{a(r)\delta}{4}y_i(r)^4.
 \end{split}
\end{equation}
This functional is obtained from a weighted Pohozaev ansatz.  The
functions $a,b,c$ are chosen, uniquely up to a common multiplicative
constant, so that differentiation along \eqref{3-6} cancels
the terms involving $(y_i')^2$, $y_i y_i'$, and $y_i^4$.  Thus
$J_i'=G y_i^2$, with the same function $G$ for both components.
\begin{lemma}\label{lemma4-1}
Let $N\in\{2,3\}$, $m:=N-1$, $\delta\in\R$, and
$\mathcal V\in C^1((0,\infty))$.  Suppose that
$y_i\in C^2((0,\infty))$ satisfies \eqref{3-6} for some
$i\in\{1,2\}$, and define $a,b,c$ and $J_i$ by \eqref{4-1} and
\eqref{4-2}.  Then $J_i\in C^1((0,\infty))$ and
\begin{equation}\label{4-3}
 J_i'(r)=G(r)y_i(r)^2,
\end{equation}
where
\begin{equation}\label{4-4}
 \frac{G(r)}{a(r)}
 =-\frac12\mathcal V'(r)
  -\frac{m}{3r}\mathcal V(r)
  +\frac{m(3-m)(2m-3)}{27r^3}.
\end{equation}
\end{lemma}
\begin{proof}
For brevity, write $y:=y_i$. From \eqref{3-6} we have
\[
 y''=-\frac{m}{r}y'+\mathcal V y-\delta y^3.
\]
Differentiating \eqref{4-2}, we obtain
\begin{align*}
 J_i'
 ={}&\frac{a'}2(y')^2+ay'y''
     +b'yy'+b(y')^2+byy''\\
 &+\frac{c'}2y^2+cyy'
     -\frac{(a\mathcal V)'}2y^2-a\mathcal V yy'+\frac{\delta a'}4y^4+a\delta y^3y'.
\end{align*}
Substituting the expression for $y''$ from
\eqref{3-6}, we obtain 
\[
 ay'y''
 =-\frac{ma}{r}(y')^2+a\mathcal V yy'-a\delta y^3y',
\]
and
\[
 byy''
 =-\frac{mb}{r}yy'+b\mathcal V y^2-b\delta y^4.
\]
Substituting these identities and collecting like terms gives
\begin{equation}\label{4-5}
 \begin{aligned}
 J_i'
 ={}&
 \left(\frac{a'}2-\frac{ma}{r}+b\right)(y')^2
 +\left(b'-\frac{mb}{r}+c\right)yy'\\
 &+\left(
      b\mathcal V+\frac{c'}2
      -\frac{(a\mathcal V)'}2
   \right)y^2
 +\delta\left(\frac{a'}4-b\right)y^4.
 \end{aligned}
\end{equation}

Since $a'=\frac{4ma}{3r}$ and $b=\frac{ma}{3r}$, 
we obtain
\[
 \frac{a'}2-\frac{ma}{r}+b
 =\frac{2ma}{3r}-\frac{ma}{r}+\frac{ma}{3r}=0
\]
and
\[
 \frac{a'}4-b
 =\frac{ma}{3r}-\frac{ma}{3r}=0.
\]
Moreover,
\[
 b'
 =\frac{m}{3}\left(\frac{a'}r-\frac{a}{r^2}\right)
 =\frac{m(4m-3)}{9r^2}a,
\]
which implies
\begin{align*}
 b'-\frac{mb}{r}+c
 &=
 \frac{m(4m-3)}{9r^2}a
 -\frac{m^2}{3r^2}a
 +\frac{m(3-m)}{9r^2}a=0.
\end{align*}
Therefore, only the terms involving $y^2$ are left.  Using
\[
 b-\frac{a'}2=-\frac{ma}{3r},
 \qquad
 \frac{c'}2
 =\frac{m(3-m)(2m-3)}{27r^3}a,
\]
we deduce
\begin{align*}
 b\mathcal V+\frac{c'}2-\frac{(a\mathcal V)'}2
 &=-\frac{a}{2}\mathcal V'
   +\left(b-\frac{a'}2\right)\mathcal V
   +\frac{c'}2\\
 &=a\left[
   -\frac12\mathcal V'
   -\frac{m}{3r}\mathcal V
   +\frac{m(3-m)(2m-3)}{27r^3}
   \right]=G.
\end{align*}
Consequently, \eqref{4-5} reduces to $J_i'=Gy_i^2$, 
which proves \eqref{4-3}.
\end{proof}

Define the functions
\begin{equation}\label{4-6}
 \mathscr J(r):=\theta_1J_1(r)+\theta_2J_2(r),
 \qquad
 \mathscr K(r):=\mathscr J(r)-\frac{\varepsilon a(r)}{4}P(r)^2.
\end{equation}
Then, we have
\begin{lemma}
\label{lemma4-2}
Let $N\in\{2,3\}$ and $0<\beta<\min\{\mu _1,\mu _2\}$.  Suppose that
$(y_1,y_2)\in\Hrad(\R^N)\times\Hrad(\R^N)$ is a positive solution of
\eqref{3-5}.  Set $m:=N-1$, with all coefficients and
functionals defined by
\eqref{3-1}, \eqref{3-4}, \eqref{4-1},
\eqref{4-2}, and \eqref{4-6}.  Then
$\mathscr K\in C^1((0,\infty))$ and, for every $r>0$,
\begin{equation}\label{4-7}
 \mathscr K'(r)
 =\frac{ma(r)P(r)}{3r}
  \left(\frac{(3-m)(2m-3)}{9r^2}-1\right).
\end{equation}
\end{lemma}

\begin{proof}
Since $\mathcal V'(r)=-\varepsilon P'(r)$, Lemma~\ref{lemma4-1} gives
\begin{equation}\label{4-8}
 \mathscr J'
 =aP\left[
  \frac{\varepsilon}{2}P'
  -\frac{m}{3r}(1-\varepsilon P)
  +\frac{m(3-m)(2m-3)}{27r^3}
 \right].
\end{equation}
On the other hand, $a'=\frac{4ma}{3r}$, and 
$ \left(\frac{\varepsilon a}{4}P^2\right)'
 =\frac{\varepsilon a}{2}PP'
  +\frac{\varepsilon ma}{3r}P^2$. 
Thus,
\begin{align*}
 \mathscr K'
 &=\mathscr J'
   -\left(\frac{\varepsilon a}{4}P^2\right)'\\
 &=aP\left[
      \frac{\varepsilon}{2}P'
      -\frac{m}{3r}(1-\varepsilon P)
      +\frac{m(3-m)(2m-3)}{27r^3}
     \right]
   -\frac{\varepsilon a}{2}PP'
   -\frac{\varepsilon ma}{3r}P^2\\
 &=-\frac{ma}{3r}P
   +\frac{m(3-m)(2m-3)}{27r^3}aP=\frac{maP}{3r}
   \left(
    \frac{(3-m)(2m-3)}{9r^2}-1
   \right),
\end{align*}
which proves \eqref{4-7}.
\end{proof}

\begin{lemma}
\label{lemma4-3}
Let $N\in\{2,3\}$ and $0<\beta<\min\{\mu _1,\mu _2\}$.  Suppose that
$(y_1,y_2)\in\Hrad(\R^N)\times\Hrad(\R^N)$ is a positive solution of
\eqref{3-5}.  Set $m:=N-1$ and $Y_i:=y_i(0)$, and define
$a,b,c,P,J_i,\mathscr J$, and $\mathscr K$ by \eqref{4-1},
\eqref{3-4}, \eqref{4-2}, and \eqref{4-6}.  Let
\[
 P_0:=\theta _1Y_1^2+\theta _2Y_2^2>0.
\]
Then
\begin{equation}\label{4-9}
 \mathscr K(r)
 =\frac{m(3-m)}{18}\frac{a(r)}{r^2}P_0+O(a(r))
 =\frac{a(r)}{9r^2}P_0+O(a(r))
 \quad\mbox{as}~~r\to 0,
\end{equation}
and
\begin{equation}\label{4-10}
 J_i(r)\to0,
 \qquad a(r)P(r)^2\to0,
 \qquad \mathscr K(r)\to0
 \quad\mbox{as}~~r\to\infty.
\end{equation}
\end{lemma}
\begin{proof}
As $r\to 0$, Lemma~\ref{lemma3-2} gives
\[
 y_i=Y_i+O(r^2),
 \qquad
 y_i'=O(r).
\]
Hence, we have
\[
 \begin{gathered}
 P=\theta _1y_1^2+\theta _2y_2^2
   =P_0+O(r^2),\\
 \sum_{i=1}^2\theta_i(y_i')^2=O(r^2),
 \qquad
 \sum_{i=1}^2\theta_i y_i y_i'=O(r),\\
 \sum_{i=1}^2\theta_i y_i^4=O(1),
 \qquad
 \mathcal V=1-\varepsilon P=O(1).
 \end{gathered}
\]
Using \eqref{4-2} and \eqref{4-6}, we can write
\begin{align*}
 \mathscr K
 =
 \frac a2\sum_{i=1}^2\theta_i(y_i')^2
 +b\sum_{i=1}^2\theta_i y_i y_i'
 +\frac c2P
 -\frac a2\mathcal V P+\frac{a\delta}{4}\sum_{i=1}^2\theta_i y_i^4
 -\frac{\varepsilon a}{4}P^2.
\end{align*}
Since $ b=\frac{ma}{3r}$ and $
 c=\frac{m(3-m)}{9r^2}a$, 
the preceding estimates yield
\begin{align*}
 \mathscr K
 ={}&
 O(ar^2)
 +\frac{ma}{3r}O(r)
 +\frac{m(3-m)}{18}\frac a{r^2}
     \bigl(P_0+O(r^2)\bigr)\\
 &+O(a)+O(a)+O(a)\\
 ={}&
 \frac{m(3-m)}{18}\frac a{r^2}P_0+O(a).
\end{align*}
Because $m(3-m)=2$ for $m\in\{1,2\}$, 
we obtain
\[
 \mathscr K(r)
 =\frac{a(r)}{9r^2}P_0+O(a(r))\quad \mbox{as}~~r\to 0,
\]
which proves \eqref{4-9}.

Since $a(r)=r^{\frac{4m}{3}}$, the preceding expansion reads
\[
 \mathscr K(r)=
 \begin{cases}
  \displaystyle \frac{P_0}{9}r^{-\frac23}+O(r^{\frac43}),
       & N=2,\\[2mm]
  \displaystyle \frac{P_0}{9}r^{\frac23}+O(r^{\frac83}),
       & N=3,
 \end{cases}
 \qquad r\to0.
\]
In particular,
\[
 \lim_{r\to0}\mathscr K(r)=+\infty
 \quad\text{if }N=2,
 \qquad
 \lim_{r\to 0}\mathscr K(r)=0
 \quad\text{if }N=3.
\]

Next, as $r\to\infty$, Lemma~\ref{lemma3-2} gives
\[
 |y_i(r)|+|y_i'(r)|
 \leq Ce^{-\kappa r}.
\]
It follows that
\[
 P(r)=O(e^{-2\kappa r}),
 \qquad
 \mathcal V(r)=1+O(e^{-2\kappa r}).
\]
Moreover, for $m\in\{1,2\}$ and $r\geq1$,
\[
 |a(r)|+|b(r)|+|c(r)|
 \leq C\bigl(1+r^{\frac{8}{3}}\bigr).
\]
Therefore, using \eqref{4-2},
\[
 |J_i(r)|
 \leq
 C\bigl(1+r^{\frac{8}{3}}\bigr)e^{-2\kappa r}
 +Cr^{\frac{8}{3}}e^{-4\kappa r}
 \to0
 \qquad \mbox{as}~~~r\to\infty.
\]
Similarly,
\[
 a(r)P(r)^2
 =O\left(r^{\frac{8}{3}}e^{-4\kappa r}\right)
 \to0 \qquad
 \mbox{as}~~~r\to\infty.
\]
Finally, since
\[
 \mathscr K
 =\theta _1J_1+\theta _2J_2
  -\frac{\varepsilon a}{4}P^2,
\]
we conclude that $ \mathscr K(r)\to0$ as $r\to\infty$, which proves \eqref{4-10}.
\end{proof}

\begin{proposition}
\label{prop4-4}
Let $N\in\{2,3\}$ and $0<\beta<\min\{\mu _1,\mu _2\}$.  Suppose that
$(y_1,y_2)\in\Hrad(\R^N)\times\Hrad(\R^N)$ is a positive solution of
\eqref{3-5}.  Set $m:=N-1$ and define
$a,b,c,P,J_i,\mathscr J$, and $\mathscr K$ by \eqref{4-1},
\eqref{3-4}, \eqref{4-2}, and \eqref{4-6}.  Then, for every
$r>0$,
\begin{equation}\label{4-11}
 \mathscr K(r)>0,
 \qquad
 \mathscr J(r)
 =\mathscr K(r)+\frac{\varepsilon a(r)}4P(r)^2>0.
\end{equation}
\end{proposition}

\begin{proof}
If $N=2$, then $m=1$.  Lemma~\ref{lemma4-2} indicates 
\[
 \mathscr K'
 =-\frac{aP}{3r}\left(1+\frac{2}{9r^2}\right)<0.
\]
Since $\lim_{r\to \infty}\mathscr K(r)=0$, integration over $(r,\infty)$
gives
\[
 \mathscr K(r)=-\int_r^\infty\mathscr K'(s)\dd s>0.
\]

If $N=3$, then $m=2$, and
\[
 \mathscr K'
 =\frac{2aP}{3r}\left(\frac{1}{9r^2}-1\right).
\]
It follows that $\mathscr K'(r)>0$ on $(0,\frac{1}{3})$ and
$\mathscr K'(r)<0$ on $(\frac{1}{3},\infty)$.  Then,  \eqref{4-9} gives $\lim_{r\to 0}\mathscr K(r)=0$. Combined with
$\lim_{r\to\infty}\mathscr K(r)=0$, for $0<r\leq\frac{1}{3}$ we have
\[
 \mathscr K(r)=\int_0^r\mathscr K'(s)\dd s>0,
\]
whereas for $r\geq\frac13$ we have
\[
 \mathscr K(r)=-\int_r^\infty\mathscr K'(s)\dd s>0.
\]
Thus $\mathscr K(r)>0$ in both dimensions. Finally,
$\varepsilon>0$, $a>0$, and $P>0$, so \eqref{4-6} gives
$\mathscr J>0$. The proof is thus completed.
\end{proof}

\section{Analysis of the ratio of the two components}\label{section5}

\begin{lemma}
\label{lemma5-1}
Let $N\in\{2,3\}$ and $0<\beta<\min\{\mu _1,\mu _2\}$.  Suppose that
$(y_1,y_2)\in\Hrad(\R^N)\times\Hrad(\R^N)$ is a positive solution of
\eqref{3-5}, and let $Y_i:=y_i(0)$.  If $Y_1=Y_2$, then
$y_1(r)=y_2(r)$ for every $r\geq0$.
\end{lemma}

\begin{proof}
Let $d(r):=y_1(r)-y_2(r)$.  Subtracting the two equations in
\eqref{3-6} gives
\begin{equation}\label{5-1}
\begin{aligned}
 \bigl(r^m d'(r)\bigr)'=r^mB(r)d(r),\\
\end{aligned}
\end{equation}
where
\begin{equation*}
 B(r):=\mathcal V(r)
       -\delta\bigl(y_1(r)^2+y_1(r)y_2(r)+y_2(r)^2\bigr).
\end{equation*}
The regular centre data are $d(0)=d'(0)=0$. In addition, the expansion $d'(r)=O(r)$ from Lemma~\ref{lemma3-2} gives
$ \lim_{r\to 0}r^m d'(r)=0$. 

 Fix $\rho_0>0$ and set
$M:=\sup\limits_{0\leq r\leq\rho_0}|B(r)|<\infty$.  For
$0<\rho\leq\rho_0$, denote by
\[
 M_d(\rho):=\sup_{0\leq s\leq\rho}|d(s)|.
\]
For $0<r\leq\rho$, integration of \eqref{5-1} from $0$ to $r$
yields
\[
 |r^md'(r)|
 \leq M M_d(\rho)\int_0^r s^m \dd s
 =\frac{M}{m+1}M_d(\rho)r^{m+1}.
\]
Thus
\[
 |d'(r)|\leq\frac{Mr}{m+1}M_d(\rho),
 \qquad
 |d(r)|\leq\frac{Mr^2}{2(m+1)}M_d(\rho).
\]
Choose $\rho>0$ such that $\frac{M\rho^2}{2(m+1)}<1$.  Taking the supremum over
$r\in[0,\rho]$ forces $M_d(\rho)=0$.  Since $d$ vanishes on $[0,\rho]$,
uniqueness for the regular initial-value problem at $r=\rho$ gives
$d(r)=0$ for every $r\geq0$.
\end{proof}

We now turn to the ratio of the two components in the common-potential system. The following lemma describes its behavior near the origin. 
\begin{lemma}
\label{lemma5-2}
Let $N\in\{2,3\}$ and $0<\beta<\min\{\mu _1,\mu _2\}$.  Suppose that
$(y_1,y_2)\in\Hrad(\R^N)\times\Hrad(\R^N)$ is a positive solution of
\eqref{3-5}, let $Y_i:=y_i(0)$, and assume $Y_2>Y_1$.  Set
$m:=N-1$ and define
\begin{equation}\label{5-2}
 \eta(r):=\frac{y_2(r)}{y_1(r)},
 \qquad \eta(0)=\frac{Y_2}{Y_1}>1,
 \qquad r\geq0.
\end{equation}
Then
\begin{equation}\label{5-3}
 \eta'(0)=0,
 \qquad
 \eta''(0)
 =-\frac{\delta}{N}\eta(0)(Y_2^2-Y_1^2)<0.
\end{equation}
In particular, $\eta'(r)<0$ for sufficiently small $r>0$.  Moreover,
\begin{equation}\label{5-4}
 \bigl(r^my_1(r)^2\eta'(r)\bigr)'
 =-\delta r^my_1(r)^4\eta(r)\bigl(\eta(r)^2-1\bigr),
\end{equation}
and
\begin{equation}\label{5-5}
\lim_{r\to 0} r^my_1(r)^2\eta'(r)=0.
\end{equation}
\end{lemma}
\begin{proof}
Since $y_1(0)=Y_1>0$, the function $\eta=\frac{y_2}{y_1}$ is smooth near the
origin.  Using $y_1'(0)=y_2'(0)=0$, we obtain
\[
 \eta'(0)
 =\frac{y_2'(0)Y_1-Y_2y_1'(0)}{Y_1^2}
 =0.
\]
Differentiating the identity $y_2=\eta y_1$ twice gives
$ y_2''
 =\eta''y_1+2\eta'y_1'+\eta y_1''$. 
Evaluating this identity at $r=0$, we obtain
\[
 \eta''(0)
 =\frac{y_2''(0)-\eta(0)y_1''(0)}{Y_1}.
\]
By Lemma~\ref{lemma3-2}, we have
\[
 y_i''(0)
 =\frac{\mathcal V(0)Y_i-\delta Y_i^3}{N},
 \qquad
 \eta(0)=\frac{Y_2}{Y_1}.
\]
Therefore,
\begin{align*}
 \eta''(0)
 =
 \frac{\eta(0)}{N}
 \left[
   \mathcal V(0)-\delta Y_2^2
   -\bigl(\mathcal V(0)-\delta Y_1^2\bigr)
 \right]=-\frac{\delta}{N}\eta(0)(Y_2^2-Y_1^2)<0.
\end{align*}
This proves \eqref{5-3}.

We next derive the flux identity.  Since $y_2=\eta y_1$, it follows that
\[
 y_2'=\eta'y_1+\eta y_1',
 \qquad
 y_2''=\eta''y_1+2\eta'y_1'+\eta y_1''.
\]
Subtracting $\eta$ times the equation for $y_1$ in
\eqref{3-6} from the equation for $y_2$, we obtain
\begin{align*}
 0
 &=
 \left(
   y_2''+\frac mr y_2'-\mathcal V y_2+\delta y_2^3
 \right)-\eta
 \left(
   y_1''+\frac mr y_1'-\mathcal V y_1+\delta y_1^3
 \right)\\
 &=
 y_1\eta''
 +\left(2y_1'+\frac mr y_1\right)\eta'
 +\delta y_1^3\eta(\eta^2-1).
\end{align*}
Dividing by the positive function $y_1$, we have
\[
 \eta''
 +\left(\frac mr+2\frac{y_1'}{y_1}\right)\eta'
 +\delta y_1^2\eta(\eta^2-1)=0.
\]
Multiplying by $r^my_1^2$  on both sides of the above equation, we can derive \eqref{5-4}.

Finally, by \eqref{3-7} we have 
$ y_i(r)=Y_i+O(r^2)$ and $y_i'(r)=O(r)$. 
Hence
\begin{align*}
 \eta'(r)
 &=
 \frac{y_2'(r)y_1(r)-y_2(r)y_1'(r)}{y_1(r)^2}\\
 &=
 \frac{
   O(r)\bigl(Y_1+O(r^2)\bigr)
   -\bigl(Y_2+O(r^2)\bigr)O(r)
 }{
   \bigl(Y_1+O(r^2)\bigr)^2
 }=O(r).
\end{align*}
Consequently,
\[
 r^my_1(r)^2\eta'(r)
 =r^m\bigl(Y_1^2+O(r^2)\bigr)O(r)
 =O(r^{m+1})\to0\quad\mbox{as}~~r\to 0,
\]
which proves \eqref{5-5}.
\end{proof}

We proceed to analyze a key new function $Z$ (see \eqref{5-6}), which plays an important role in our proof.  
\begin{lemma}
\label{lemma5-3}
Let $N\in\{2,3\}$ and $0<\beta<\min\{\mu _1,\mu _2\}$.  Suppose that
$(y_1,y_2)\in\Hrad(\R^N)\times\Hrad(\R^N)$ is a positive solution of
\eqref{3-5}, let $Y_i:=y_i(0)$, and assume $Y_2>Y_1$.  Set
$m:=N-1$.  Define
$a,b,c,P,J_i$, and $\mathscr J$ by \eqref{4-1}, \eqref{3-4},
\eqref{4-2}, and \eqref{4-6}.  Let $\eta$ be given by
\eqref{5-2}, and define, for $r>0$,
\begin{equation}\label{5-6}
 X(r):=\eta(r)^2J_1(r)-J_2(r),
 \quad
 \mathscr D(r):=\theta _1+\theta _2\eta(r)^2>0,
 \quad
 Z(r):=\frac{X(r)}{\mathscr D(r)}.
\end{equation}
Then
\begin{equation}\label{5-7}
 X'(r)=2\eta(r)\eta'(r)J_1(r),
 \qquad
 Z'(r)=\frac{2\eta(r)\eta'(r)}{\mathscr D(r)^2}\mathscr J(r).
\end{equation}
Furthermore, we have
\begin{equation}\label{5-8}
 \begin{split}
 X=-a\eta\eta' y_1y_1'
       -\frac a2(\eta')^2y_1^2
       -b\eta\eta'y_1^2+\frac{a\delta}{4}\eta^2y_1^4(1-\eta^2),
 \end{split}
\end{equation}
and
\begin{equation}\label{5-9}
 \lim_{r\to 0}X(r)=0,
 \qquad \lim_{r\to 0}Z(r)=0.
\end{equation}
If $\eta$ is nonincreasing on $(0,\infty)$, then 
\begin{equation}\label{5-10}
 \lim_{r\to\infty}Z(r)=0.
\end{equation}
\end{lemma}
\begin{proof}
By Lemma~\ref{lemma4-1} and the identity $y_2=\eta y_1$, we have
\[ J_2'=Gy_2^2=G\eta^2y_1^2=\eta^2J_1'. \]
Hence,
\[  X'=\bigl(\eta^2J_1-J_2\bigr)'=2\eta\eta'J_1+\eta^2J_1'-J_2'=2\eta\eta'J_1.\]
Moreover, since $X=\eta^2J_1-J_2$, we have
$ J_2=\eta^2J_1-X$. 
It follows that
\begin{align*}
 \mathscr J&=\theta _1J_1+\theta _2J_2=\theta _1J_1+\theta _2(\eta^2J_1-X)\\
 &=(\theta _1+\theta _2\eta^2)J_1-\theta _2X=\mathscr D J_1-\theta _2X.
\end{align*}
Since $
 \mathscr D'=(\theta _1+\theta _2\eta^2)'
 =2\theta _2\eta\eta'$, we deduce that
\begin{align*}
 Z'&=\frac{X'\mathscr D-X\mathscr D'}{\mathscr D^2}=\frac{
   2\eta\eta'J_1\mathscr D
   -2\theta _2\eta\eta'X
 }{\mathscr D^2}\\
 &=\frac{2\eta\eta'}{\mathscr D^2}   \bigl(\mathscr D J_1-\theta _2X\bigr)=\frac{2\eta\eta'}{\mathscr D^2}\mathscr J.
\end{align*}
This proves \eqref{5-7}.

We next derive the exact formula for $X$.  Denoting by $y:=y_1$, we have 
$ y_2=\eta y$ and $y_2'=\eta'y+\eta y'$. 
By \eqref{4-2},
\begin{align*}
 \eta^2J_1
 ={}&
 \frac a2\eta^2(y')^2
 +b\eta^2yy'
 +\frac c2\eta^2y^2
 -\frac a2\mathcal V\eta^2y^2
 +\frac{a\delta}{4}\eta^2y^4,
\end{align*}
whereas
\begin{align*}
 J_2 = \frac a2(\eta'y+\eta y')^2
 +b\eta y(\eta'y+\eta y')
 +\frac c2\eta^2y^2-\frac a2\mathcal V\eta^2y^2
 +\frac{a\delta}{4}\eta^4y^4.
\end{align*}
Therefore, 
\begin{align*}
 X = \frac a2
 \left[   \eta^2(y')^2-(\eta'y+\eta y')^2
 \right]+b\left[\eta^2yy'-\eta y(\eta'y+\eta y')
 \right]
 +\frac{a\delta}{4}(\eta^2-\eta^4)y^4.
\end{align*}
Expanding the two squared and mixed terms gives
\begin{align*}
 \eta^2(y')^2-(\eta'y+\eta y')^2
 &=-2\eta\eta'yy'-(\eta')^2y^2,\\
 \eta^2yy'-\eta y(\eta'y+\eta y')
 &=-\eta\eta'y^2.
\end{align*}
Consequently,
\[ X=-a\eta\eta'yy'-\frac a2(\eta')^2y^2-b\eta\eta'y^2+\frac{a\delta}{4}\eta^2y^4(1-\eta^2),
\]
which is \eqref{5-8}.

As $r\to 0$, \eqref{3-7} and
\eqref{5-3} give
\[
 \begin{gathered}
 y(r)=Y_1+O(r^2),
 \qquad y'(r)=O(r),\\
 \eta(r)=\eta(0)+O(r^2),
 \qquad \eta'(r)=O(r).
 \end{gathered}
\]
Since $a=r^{\frac{4m}{3}}$ and $ b=\frac{ma}{3r}$, the four terms in \eqref{5-8} satisfy
\begin{align*}
 a\eta\eta'yy'&=O(ar^2),\qquad a(\eta')^2y^2=O(ar^2),\\
 b\eta\eta'y^2 &=\frac{ma}{3r}O(r)=O(a),\qquad a\eta^2y^4(1-\eta^2)=O(a).
\end{align*}
Thus, we obtain
\[
 X(r)=O(a(r))=O(r^{\frac{4m}{3}})\to 0\quad \mbox{as}~~~r\to 0.
\]
Furthermore, 
\[
 \mathscr D(r)
 =\theta _1+\theta _2\eta(0)^2+O(r^2),\qquad \theta_1+\theta_2\eta(0)^2>0.
\]
As a result, 
\[  Z(r)=\frac{X(r)}{\mathscr D(r)}=O(r^{\frac{4m}{3}})\to0\quad \mbox{as}~~~r\to 0.\]
This proves \eqref{5-9}.

Finally, suppose that $\eta$ is nonincreasing.  Since $\eta>0$,
\[
 0<\eta(r)\leq\eta(0),
 \qquad
 \mathscr D(r)=\theta _1+\theta _2\eta(r)^2\geq\theta _1.
\]
Therefore, we obtain
\begin{align*}
 |Z(r)|\leq
 \frac{
   \eta(r)^2|J_1(r)|+|J_2(r)|
 }{\mathscr D(r)}\leq
 \frac{
   \eta(0)^2|J_1(r)|+|J_2(r)|
 }{\theta _1}.
\end{align*}
By Lemma~\ref{lemma4-3}, $\lim_{r\to\infty}J_1(r)=\lim_{r\to\infty}J_2(r)=0$. Thus, $\lim_{r\to\infty}|Z(r)|=0$, which proves \eqref{5-10}.
\end{proof}

\begin{lemma}
\label{lemma5-4}
Let $N\in\{2,3\}$ and $0<\beta<\min\{\mu _1,\mu _2\}$. Then there is no positive
radial solution
$(y_1,y_2)\in\Hrad(\R^N)\times\Hrad(\R^N)$ of
\eqref{3-5} such that $y_2(0)>y_1(0)$.
\end{lemma}
\begin{proof}
Suppose, for contradiction, that such a solution exists.  Set
\[
 Y_i:=y_i(0),
 \qquad
 \eta:=\frac{y_2}{y_1}.
\]
Since $Y_2>Y_1$, we have $\eta(0)=\frac{Y_2}{Y_1}>1$.  It follows from
\eqref{5-3} that $\eta'(0)=0$ and $\eta''(0)<0$.  Hence, for all
sufficiently small $r>0$,
\[
 \eta'(r)=\eta''(0)r+o(r)<0.
\]

We first verify that $Z'$ is integrable at the origin.  By
\eqref{5-3},
\[
 \eta'(r)=O(r),
 \qquad
 \mathscr D(r)\to
 \theta _1+\theta _2\eta(0)^2>0
 \quad\text{as }r\to 0.
\]
Recalling \eqref{4-6}, the expansion \eqref{4-9} gives
\[
 \mathscr J(r)=
 \begin{cases}
  O(r^{-\frac{2}{3}}),&N=2,\\
  O(r^{\frac{2}{3}}),&N=3,
 \end{cases}
 \qquad\text{as }r\to 0.
\]
Therefore, \eqref{5-7} yields
\[
 Z'(r)=
 \begin{cases}
  O(r^{\frac{1}{3}}),&N=2,\\
  O(r^{\frac{5}{3}}),&N=3,
 \end{cases}
 \qquad\text{as }r\to 0.
\]
Hence, $Z'$ is integrable at the origin.

We now claim that $\eta'$ has no zero.  Suppose otherwise and set
\[
 \mathcal Z_\eta:=\{r>0:\eta'(r)=0\}\ne\varnothing,
 \qquad
 r_*:=\inf\mathcal Z_\eta.
\]
Since $\eta'(r)<0$ for all sufficiently small $r>0$, there exists $r_0>0$
such that $\eta'(r)<0$ for $0<r\leq r_0$.  Hence $r_*\geq r_0>0$.  By the
definition of the infimum, for every $n\geq1$ one can choose
$r_n\in\mathcal Z_\eta$ such that
\[
 r_*\leq r_n<r_*+\frac1n.
\]
Thus $r_n\to r_*$, and continuity of $\eta'$ gives $\eta'(r_*)=0$.
Moreover, $\eta'$ has no zero on $(0,r_*)$.  Since it is negative near the
origin, continuity forces it to remain negative throughout this interval.
Consequently,
\begin{equation}\label{5-11}
 \eta'(r)<0\quad\text{for }0<r<r_*,
 \qquad
 \eta'(r_*)=0.
\end{equation}

We claim that
\begin{equation}\label{5-12}
 \eta(r_*)<1.
\end{equation}
Indeed, suppose that $\eta(r_*)\geq1$.  Since $\eta'(r)<0$ for $r\in(0,r_*)$, we have
\[  \eta(r)>\eta(r_*)\geq 1 \qquad \forall 0<r<r_*,\]
and hence
\[  y_1(r)^4\eta(r)\bigl(\eta(r)^2-1\bigr)>0 \qquad \forall 0<r<r_*.\]
Integrating \eqref{5-4} over $(0,r_*)$ and using
\eqref{5-5}, we obtain
\begin{align*}
 r_*^my_1(r_*)^2\eta'(r_*)&=
 \lim_{\rho\to 0}
 \rho^my_1(\rho)^2\eta'(\rho) -\delta\int_0^{r_*}
 s^my_1(s)^4\eta(s)\bigl(\eta(s)^2-1\bigr) \dd s\\
 &=-\delta\int_0^{r_*}s^my_1(s)^4\eta(s)\bigl(\eta(s)^2-1\bigr)\dd s<0.
\end{align*}
On the other hand, \eqref{5-11} gives $r_*^my_1(r_*)^2\eta'(r_*)=0$, 
which is a contradiction.  Thus \eqref{5-12} holds.

At $r=r_*$, the identity \eqref{5-8} and
$\eta'(r_*)=0$ yield
\begin{equation}
\begin{aligned}
 X(r_*) 
&= \frac{a(r_*)\delta}{4}\eta(r_*)^2y_1(r_*)^4 \bigl(1-\eta(r_*)^2\bigr)>0.
\end{aligned}
\end{equation}
Here, we used the facts that
\[  a(r_*)>0,\qquad \delta>0,\qquad \eta(r_*)>0,\qquad y_1(r_*)>0,\qquad 1-\eta(r_*)^2>0.\]
Since $\mathscr D(r_*)=\theta _1+\theta _2\eta(r_*)^2>0$, we conclude that
\begin{equation}\label{5-14}
 Z(r_*)=\frac{X(r_*)}{\mathscr D(r_*)}>0.
\end{equation}

On the other hand, Proposition~\ref{prop4-4} gives
$\mathscr J(r)>0$ for every $r>0$.  Together with
\eqref{5-11} and \eqref{5-7}, this yields
\[
 Z'(r)=\frac{2\eta(r)\eta'(r)}{\mathscr D(r)^2}\mathscr J(r)<0
 \qquad\text{for every }0<r<r_*.
\]
For every $0<\rho<r_*$, we have
\[
 \begin{aligned}
  Z(r_*)-Z(\rho)=\int_\rho^{r_*}Z'(s)\dd s=\int_\rho^{r_*}
    \frac{2\eta(s)\eta'(s)}{\mathscr D(s)^2}
    \mathscr J(s)\dd s.
 \end{aligned}
\]
Letting $\rho\to 0$, using \eqref{5-9} and the integrability
established above, we obtain
\[
 Z(r_*)
 =\int_0^{r_*}
   \frac{2\eta(s)\eta'(s)}{\mathscr D(s)^2}
   \mathscr J(s)\dd s<0.
\]
This contradicts \eqref{5-14}.  Hence $\eta'$ has no zero.
Since it is negative near the origin, we conclude that
\begin{equation}\label{5-15}
 \eta'(r)<0\qquad \forall r>0.
\end{equation}

Applying \eqref{5-7} once more, together with
\eqref{5-15} and $\mathscr J>0$, we get
\[
 Z'(r)<0,\qquad\text{for every }r>0.
\]
In particular, for every fixed $r_0>0$ and all $r\geq r_0$,
\[
 \begin{aligned}
 Z(r)\leq Z(r_0)=\int_0^{r_0}
   \frac{2\eta(s)\eta'(s)}{\mathscr D(s)^2}
   \mathscr J(s)\dd s<0.
 \end{aligned}
\]
However, \eqref{5-15} implies that $\eta$ is
nonincreasing.  Hence Lemma~\ref{lemma5-3} gives
$$ \lim\limits_{r\to\infty}Z(r)=0.$$
Sending $r\to\infty$ in the preceding inequality yields
\[
 0=\lim_{r\to\infty}Z(r)\leq Z(r_0)<0,
\]
which is impossible.  This completes the proof.
\end{proof}
Combining all the discussion above, we can now prove the following proposition.
\begin{proposition}
\label{prop5-5}
Let $N\in\{2,3\}$ and $0<\beta<\min\{\mu _1,\mu _2\}$.  If
$(y_1,y_2)\in\Hrad(\R^N)\times\Hrad(\R^N)$ is a positive radial 
solution of \eqref{3-5}, then
\begin{equation}\label{5-16}
 y_1(r)=y_2(r)\qquad\text{for every }r\geq0.
\end{equation}
\end{proposition}
\begin{proof}
Let us put $Y_1:=y_1(0)$ and $Y_2:=y_2(0)$.  We distinguish three cases.

\medskip
\noindent
\emph{\textbf{Case 1}: $Y_1=Y_2$.}
By Lemma~\ref{lemma5-1},
\[
 y_1(r)=y_2(r)
 \qquad\forall r\geq0.
\]
\medskip
\noindent
\emph{\textbf{Case 2:} $Y_2>Y_1$.}
This is impossible by Lemma~\ref{lemma5-4}.

\medskip
\noindent
\emph{\textbf{Case 3:} $Y_1>Y_2$.} Interchange the triples $(\mu_1,y_1,\theta_1)$ and $(\mu_2,y_2,\theta_2)$, and denote the relabelled components by $\tilde y_1$ and $\tilde y_2$. Then Case 3 is reduced to Case 2, and we conclude that it is also impossible.

With all this understood, we can now conclude that $y_1(r)\equiv y_2(r)$ for all $r\geq 0$. 
\end{proof}

\section{Proofs of the main results}\label{section6}

In this last section, we combine our previous results to prove the main results of this paper.
\begin{proof}[Proof of Theorem~\ref{theorem1-1}]
Let $(u,v)$ be a positive radial solution of
\eqref{1-1}.  Define
\[
 y_1:=\frac{u}{\lambda _1},
 \qquad
 y_2:=\frac{v}{\lambda _2}.
\]
By Lemma~\ref{lemma3-1}, for $i\in\{1,2\}$, we have
\[ -\Delta y_i+\mathcal V y_i=\delta y_i^3, \]
where
$ P=\theta _1y_1^2+\theta _2y_2^2$, and $\mathcal V=1-\varepsilon P$. From Proposition~\ref{prop5-5}, we get
$ y_1\equiv y_2=:y$. Since $\theta_1+\theta_2=1$, we have $ P =\theta _1y^2+\theta _2y^2 =(\theta _1+\theta _2)y^2 =y^2$. 
Thus, $\mathcal V=1-\varepsilon y^2$, and the common-potential equation reduces to
\begin{align*}
 -\Delta y+y=(\delta+\varepsilon)y^3=y^3,
\end{align*}
where we used $\delta+\varepsilon=1$. Moreover, $y\in\Hrad(\R^N)$ and $y>0$ in $\R^N$. By \cite{Kwong} we obtain $y=w$.
Returning to the original variables, we obtain
\begin{align*}
 u&=\lambda _1y_1=\sqrt{\frac{\mu _2-\beta}
                 {\mu _1\mu _2-\beta^2}}w, \quad v=\lambda _2y_2 =\sqrt{\frac{\mu _1-\beta}
                 {\mu _1\mu _2-\beta^2}}w.
\end{align*}
Hence, every positive radial solution must have the form
\eqref{1-4}.

Conversely, let $w$ be the positive radial solution given by
\cite{Kwong}, and define
\[
 u_*:=\lambda _1w,
 \qquad
 v_*:=\lambda _2w.
\]
Recall that
\[
 D=\mu _1\mu _2-\beta^2,
 \qquad
 \lambda _1^2=\frac{\mu _2-\beta}{D},
 \qquad
 \lambda _2^2=\frac{\mu _1-\beta}{D}.
\]
A direct calculation gives
\begin{align*}
 \mu _1\lambda _1^2+\beta\lambda _2^2=\frac{\mu _1(\mu _2-\beta)+\beta(\mu _1-\beta)}{D}=
 \frac{\mu _1\mu _2-\beta^2}{D} =1
\end{align*}
and
\begin{align*}
 \beta\lambda _1^2+\mu _2\lambda _2^2=
 \frac{\beta(\mu _2-\beta)
       +\mu _2(\mu _1-\beta)}{D}=\frac{\mu _1\mu _2-\beta^2}{D}=1.
\end{align*}
Since $w$ satisfies $-\Delta w+w=w^3$, we have
\begin{align*}
 -\Delta u_*+u_*=\lambda _1(-\Delta w+w)=\lambda _1w^3,
\end{align*}
whereas
\begin{align*}
 \mu _1u_*^3+\beta u_*v_*^2= \mu _1\lambda _1^3w^3 +\beta\lambda _1\lambda _2^2w^3= \lambda _1 \bigl(\mu _1\lambda _1^2+\beta\lambda _2^2\bigr)w^3=\lambda _1w^3.
\end{align*}
Therefore,
\[
 -\Delta u_*+u_*
 =\mu _1u_*^3+\beta u_*v_*^2.
\]
Similarly, we have
\[
 -\Delta v_*+v_*=\mu _2v_*^3+\beta u_*^2v_*.
\]
Hence, $(u_*,v_*)$ is a positive radial solution of
\eqref{1-1}. Because the first part of the proof shows that every
positive radial  solution equals $(u_*,v_*)$, the solution is unique.
This finishes the proof.
\end{proof}

\begin{proof}[Proof of Corollary \ref{cor1-3}.]
This is a direct consequence of Theorem~\ref{theorem1-1} and \cite[Theorem 2]{BuscaSirakov}, so we omit the details.
\end{proof}

\begin{remark}
The proof requires no monotonicity assumption on either $y_i$ or
$\mathcal V$, no sign condition on $\mathcal V(0)$, and no additional
condition involving $r\mathcal V'$.  The relation
$\mathcal V=1-\varepsilon P$ cancels the two
terms involving the unknown potential in $\mathscr K'$.  Positivity of
$\mathscr J$ then relates the signs of $\eta'$ and $Z'$.
\end{remark}

\begin{remark}
When $\beta=0$, uniqueness still holds among pairs that are radial about a
prescribed common centre, but Corollary~\ref{cor1-3} fails because
the two components may be translated independently.  On the other hand, at the upper endpoint,
if $\mu _1=\mu _2=\beta=\mu$, the family
\begin{equation}
\label{6.1}
 (u,v)=\mu^{-\frac{1}{2}}(\cos\vartheta\,w,\sin\vartheta\,w),
 \qquad 0<\vartheta<\frac\pi2
\end{equation}
provides all the positive solutions. Indeed, without loss of generality we may assume that $u,~v$ are radially symmetric  with respect to $0$. Then equation \eqref{1-1} can be written as
\begin{align*}
\begin{cases}
u''+\frac{m}{r}u'-u+\mu(u^2+v^2)u=0,\\
v''+\frac{m}{r}v'-v+\mu(u^2+v^2)v=0,\\
u'(0)=v'(0)=0.
\end{cases}
\end{align*}
Then we can derive that
$$(r^m(uv'-vu'))'=0\quad\mbox{and}\quad r^m(uv'-vu')|_{r=0}=0.$$
Consequently
$$r^m(uv'-vu')\equiv 0\quad\mbox{for all} ~r>0.$$
It implies that $\frac{u}{v}\equiv C$ and all the positive solutions of \eqref{1-1} can be written as \eqref{6.1} .
Therefore, the condition $\beta>0$ is needed for
the uniqueness modulo simultaneous translation in
Corollary~\ref{cor1-3}, whereas the upper endpoint cannot be included
in the radial uniqueness statement of Theorem~\ref{theorem1-1}.
\end{remark}

\section*{\bf Acknowledgements}
Hong-Ge Chen was supported by the National Natural Science Foundation of China (Grant Nos.~12201607 and~12571249) and the Postdoctoral Project of Hubei Province (Grant No.~2024HBBHXF095). Yong Liu was supported by  National Natural Science Foundation of China No. 12471204. Juncheng Wei was supported by National R\&D Program of China (Grant No. 2022YFA1005602), and Hong Kong General Research Fund ``New frontiers in singular limits of nonlinear partial differential equation". Wen Yang was supported by the National Key Research and Development Program of China (Grant No.~2022YFA1006800), the National Natural Science Foundation of China
(Grant No. 12531010), the Science and Technology Development Fund of the Macao SAR (FDCT, Grant No.~0070/2024/RIA1), the Multi-Year Research Grants of the University
of Macau (Grant No. MYRG-GRG2025-00051-FST), and the University of Macau Development Foundation (Grant No.~TISF/2025/006/FST).
\medskip

\noindent {\bf Data availability statement}: There are no data associated with this article.
\medskip

\noindent {\bf AI assistance statement}: The authors used AI models to assist with calculations and writings; the main ideas, mathematical validation, and all final checks remain the sole responsibility of the human authors.

\end{document}